\documentclass[12pt]{amsart}
\usepackage{a4wide}
\usepackage[T1]{fontenc}
\usepackage{amssymb,amsmath,amsthm,latexsym}
\usepackage{mathrsfs}
\usepackage{xcolor}
\usepackage{euscript}
\usepackage{graphicx}
\usepackage{mdwlist}
\usepackage{enumerate}
\usepackage{mathtools,dsfont,wasysym}
\usepackage{stmaryrd}
\usepackage[all]{xy}
\usepackage{multicol}
\usepackage{hyperref}
\hypersetup{colorlinks=true,  linkcolor=blue, citecolor=red, urlcolor=cyan}

\newtheorem{theorem}{Theorem}[section]

\newtheorem{lemma}[theorem]{Lemma}

\newtheorem{proposition}[theorem]{Proposition}
\newtheorem{fact}[theorem]{Fact}

\theoremstyle{remark}
\newtheorem{remark}[theorem]{Remark}
\theoremstyle{definition}
\newtheorem{definition}[theorem]{Definition}

\newtheorem{example}[theorem]{Example}
\newtheorem{question}[theorem]{Question}
\numberwithin{equation}{section}

\newcommand{\p}{\varphi}

\newcommand{\U}{\mathcal{U}}

\renewcommand{\leq}{\leqslant}
\renewcommand{\geq}{\geqslant}

\DeclareMathOperator{\ims}{ims}
\DeclareMathOperator{\Lev}{Lev}
\DeclareMathOperator{\htte}{ht}
\DeclareMathOperator{\cf}{cf}

\newcommand{\cw}{\tau_{\rm cw}}
\newcommand{\ccw}{\tau_{\sigma\text{\rm-cw}}}

\newcounter{smallromans}

\newenvironment{romanenumerate}
{\begin{list}{{\normalfont\textrm{(\roman{smallromans})}}}%
  {\usecounter{smallromans}\setlength{\itemindent}{0cm}%
   \setlength{\leftmargin}{5.5ex}\setlength{\labelwidth}{5.5ex}%
   \setlength{\topsep}{.5ex}\setlength{\partopsep}{.5ex}%
   \setlength{\itemsep}{0.1ex}}}%
{\end{list}}

\begin{document}
\title{Weakly Corson compact trees}

\author[T.~Russo]{Tommaso Russo}
\address[T.~Russo]{Institute of Mathematics\\ Czech Academy of Sciences\\ \v{Z}itn\'a 25, 115 67 Prague 1\\ Czech Republic\\ and Department of Mathematics\\Faculty of Electrical Engineering\\Czech Technical University in Prague\\Technick\'a 2, 166 27 Prague 6\\ Czech Republic}
\email{russo@math.cas.cz, russotom@fel.cvut.cz}

\author[J.~Somaglia]{Jacopo Somaglia}
\address[J.~Somaglia]{Dipartimento di Matematica ``F. Enriques'' \\Universit\`a degli Studi di Milano\\Via Cesare Saldini 50, 20133 Milano\\Italy}
\email{jacopo.somaglia@unimi.it}

\dedicatory{In memoriam: Benedetto Braggion (1948--2021)}
\date{\today}
\thanks{T.~Russo was supported by the GA\v{C}R project 20-22230L; RVO: 67985840. J.~Somaglia was supported  by Universit\`a degli Studi di Milano. Both authors were also supported by Gruppo Nazionale per l'Analisi Matematica, la Probabilit\`a e le loro Applicazioni (GNAMPA) of Istituto Nazionale di Alta Matematica (INdAM), Italy.}

\begin{abstract} We introduce and study a new topology on trees, that we call the countably coarse wedge topology. Such a topology is strictly finer than the coarse wedge topology and it turns every chain complete, rooted tree into a Fr\'echet--Urysohn, countably compact topological space. We show the r\^{o}le of such topology in the theory of weakly Corson and weakly Valdivia compacta. In particular, we give the first example of a compact space $T$ whose every closed subspace is weakly Valdivia, yet $T$ is not weakly Corson. This answers a question due to Ond\v{r}ej Kalenda.
\end{abstract}
\maketitle

\section{Introduction}

A compact topological space is \textit{Corson} if it is homeomorphic to a subset of 
\begin{equation*}
\Sigma(\Gamma)=\{x\in [0,1]^{\Gamma}\colon|\{\gamma\in \Gamma\colon x(\gamma)\neq 0\}|\leq \omega\},
\end{equation*}
for some set $\Gamma$, where $[0,1]^\Gamma$ is endowed with the product topology. A compact space $K$ is \textit{Valdivia} if there is a homeomorphic embedding $h\colon K \to [0,1]^{\Gamma}$ such that $h^{-1}(\Sigma(\Gamma))$ is dense in $K$; in this case, $h^{-1}(\Sigma(\Gamma))$ is called a \textit{$\Sigma$-subset} of $K$. Corson and Valdivia compacta are well-investigated classes of compact spaces for their interest in descriptive topology and nonseparable Banach space theory, see, \emph{e.g.}, \cite{DG, KKL, K2, K3} and references therein. While the class of Corson compacta is a well-rounded class that is stable under continuous images \cite{MR}, Valdivia compacta lack this stability property (indeed, Valdivia compacta are not even stable under open continuous images, \cite{KuUs}). Moreover, it is easily seen that Valdivia compacta are not closed under taking closed subspaces. For these reasons, the classes of Valdivia compacta and of their continuous images are harder to investigate.

A topological space is a \textit{Corson countably compact} if it is homeomorphic to a countably compact subset of $\Sigma(\Gamma)$, for some $\Gamma$. A compact space that is a continuous image of a Corson  countably compact is a \textit{weakly Corson compact}, \cite{K1}. As it turns out, a $\Sigma$-subset of a Valdivia compactum is a Corson countably compact and, vice versa, the \v{C}ech--Stone compactification of a Corson countably compact $X$ is a Valdivia compact, having $X$ as a $\Sigma$-subset, \cite[p.~7]{K2}. Therefore, Corson countably compact spaces are tightly connected to Valdivia compacta and weakly Corson compacta to continuous images of Valdivia compacta (called \textit{weakly Valdivia} in \cite{K1}). This was the motivation for their introduction in \cite{K1}, where the following open problem is raised.

\begin{question}[Kalenda, \cite{K1}]\label{q: main question}
Let $K$ be a compact space such that each closed subset of $K$ is weakly Valdivia. Is $K$ weakly Corson?
\end{question}

The motivation behind this question is that in the context of Valdivia compacta, the space $[0,\omega_1]$ is a simple example of a compactum whose every closed subset is Valdivia, but which is not Corson. However, it turns out that ordinal segments can't provide counterexamples to Question \ref{q: main question}, \cite[Theorem 3.5]{K1} and subtler counterexamples are to be sought.\smallskip

The above question was the main motivation at the origin of our research and we do solve it in the negative in our note (Example \ref{e: binaryTree}). Our counterexample is the full binary tree of height $\omega_1 +1$, endowed with the coarse wedge topology. While the construction of the example is extremely simple, a certain amount of work is needed in order to show that it admits the desired properties. Given the need to study Corson countably compact spaces, an important part of our paper consists in the introduction of a topology, that we call the \textit{countably coarse wedge topology}, that makes every chain complete, rooted tree into a countably compact topological space. This topology is a variation of the coarse wedge topology and we also study in some detail the connection between these topologies.\smallskip

In Section \ref{s: tree} we recall basic notions on trees and we prove that any tree of height $\omega_1+1$ endowed with the coarse wedge topology is hereditarily Valdivia, improving \cite[Theorem 4.1]{S2}. The core of the article are Sections \ref{s: ccw} and \ref{s: wct}, where we introduce the countably coarse wedge topology and we show the importance of such topology for the study of weakly Corson and weakly Valdivia compact trees. Finally, Section \ref{s: gdelta} contains some remarks on $G_\delta$ points in weakly Corson compact trees.

All topological spaces are assumed to be Hausdorff and completely regular. Recall that a topological space $X$ is \textit{countably compact} if every countable open cover of $X$ has a finite subcover; equivalently, if every infinite subset of $X$ has an accumulation point. Further, $X$ is \textit{Fr\'{e}chet--Urysohn} if for every $A\subseteq X$ and $x\in \overline{A}$ there is a sequence $(x_{n})_{n\in\omega}\subseteq A$ that converges to $x$. We say that $A\subseteq X$ is \textit{countably closed} if $\overline{C}\subseteq A$ for every countable subset $C$ of $A$. Finally, $\beta X$ is the \v{C}ech--Stone compactification of a completely regular space $X$.

\section{Trees and Valdivia compacta}\label{s: tree}
A \textit{tree} is a partially ordered set $(T,\leq)$ such that the set of predecessors $\{s\in T\colon s< t\}$ of any $t\in T$ is well-ordered by $<$. A tree $T$ is said to be \textit{rooted} if it has only one minimal element (denoted by $0_T$), called \textit{root}. For a tree $T$ and $t\in T$, we write $\hat{t}=\{s\in T\colon s\leq t\}$. A subset $S$ of $T$ is an \textit{initial part} of $T$ if $\hat{t}\subseteq S$ for every $t\in S$ (in other words, $S$ is downward closed). A \textit{chain} in $T$ is a totally ordered subset of $T$; $T$ is called \textit{chain complete} if every chain has a supremum. For an element $t\in T$, $\htte(t,T)$ denotes the order type of $\{s\in T\colon s< t\}$. Given an ordinal $\alpha$, the set $\Lev_{\alpha}(T)=\{t\in T\colon \htte(t,T)=\alpha\}$ is called the \textit{$\alpha$-th level} of $T$. The \textit{height} of $T$, denoted by $\htte(T)$, is the least $\alpha$ such that $\Lev_{\alpha}(T)=\emptyset$. For an element $t\in T$, $\cf(t)$ denotes the cofinality of $\htte(t,T)$, where $\cf(t)=0$ if $\htte(t,T)$ is a successor ordinal or $\htte(t,T)=0$. We write $\cf(t,T)$ in case it is important to specify the tree $T$. Moreover, $\ims(t)=\{s\in T\colon t\leq s, \,\htte(s,T)=\htte(t,T)+1\}$ denotes the set of \textit{immediate successors} of $t$. A rooted and chain complete tree $T$ is a \textit{full binary tree} if $|\ims(t)|=2$ for every $t\in T$ with $\htte(t,T)+1 <\htte(T)$. Finally, borrowing the notation from \cite{K4} we denote $I(T)=\{t\in T\colon  \cf(t)<\omega\}$.

For $t\in T$ we put $V_t=\{s\in T\colon s\geq t\}$. The \textit{coarse wedge topology} on a tree $T$ is the one whose subbase is given by the sets $V_t$ and their complements, where $t\in I(T)$.
We denote the coarse wedge topology by $\cw$; for further information on it, we refer to \cite{K4, S2, S3}. Let us mention in passing that the coarse wedge topology also coincides with the path topology, \cite{G1, T1}. Recall that a tree $T$ is compact Hausdorff in the coarse wedge topology if and only if $T$ is chain complete and it has finitely many minimal elements, \cite[Corollary 3.5]{N1}. For this reason, from now on we will only consider chain complete, rooted trees. Importantly, this assumption assures that every two elements $s,t\in T$ admit an infimum $s\wedge t$; indeed, $s\wedge t=\max(\hat{s}\cap\hat{t})$. The following fact is standard (see, \emph{e.g.}, the proof of \cite[Lemma 3.5]{S3}).

\begin{fact}\label{f: Vt disjoint from closure}
Let $(T,\cw)$ be a tree and $t\in T$ be such that $\cf(t)\geq\omega_1$. Then $T\setminus V_t$ is countably closed.
\end{fact}

\begin{proof} Fix a countable subset $S$ of $T\setminus V_t$; we will prove that $V_t\cap \overline{S}=\emptyset$. The set $V_t\setminus\{t\}=\cup_{r\in\ims(t)}V_r$ is $\cw$-open; hence, $V_t\setminus\{t\} \cap \overline{S}=\emptyset$ and we only need to show that $t\notin \overline{S}$. For every $s\in S$, let $t_s\coloneqq t \wedge s < t$ (since $s\notin V_t$). Then the assumption $\cf(t)\geq\omega_1$ gives the existence of $p\in I(T)$ with $\sup t_s <p <t$. Hence, $V_p\cap S=\emptyset$, since $s\geq p$ would imply $t_s =t \wedge s\geq p$ which is false. Since $t\in V_p$ and $V_p$ is $\cw$-open, we are done.
\end{proof}

A tree $T$ is an \textit{$r$-tree} if it satisfies the following condition
\[
|\ims(t)|<\omega \mbox{ whenever } \cf(t)\geq\omega_1.
\]
If, in addition, $|\ims(t)|\leq 1$ whenever  $\cf(t)\geq\omega_1$, then $T$ is an \textit{$r_1$-tree}. By \cite[Theorem 3.1]{S2}, $r$-trees, endowed with the coarse wedge topology, are characterised by having a retractional skeleton\footnote{Since we won't need the notion of retractional skeleton in our paper, we shall refer to \cite{KuMi} for its definition.}. In particular, if $T$ is a Valdivia tree, then $T$ is an $r$-tree. \smallskip

It was proved in \cite[Theorem 4.1]{S2} that every tree $T$ with $\htte(T)\leq\omega_1 +1$ is Valdivia when endowed with the coarse wedge topology. Combining this with \cite[Lemma 2.1]{S3}, we obtain that, if $T$ is a tree with $\htte(T)\leq\omega_1 +1$ and $S\subseteq T$ is closed both in topology and with respect to the $\wedge$-operation, then $S$ is Valdivia as well. In Proposition \ref{p: omega1Valdivia} below we generalise this result dropping the $\wedge$-closedness assumption.

For the proof we need to recall a classical Rosenthal type characterisation of Valdivia compacta. A family $\U$ of subsets of $T$  is \textit{$T_0$-separating} in $T$ if for every two distinct elements $s,t\in  T$ there is $U\in\mathcal{U}$ satisfying $|\{s,t\}\cap U|=1$. Given $s\in T$ we denote by $\U(s)=\{U\in\U \colon s\in U\}$. A family $\mathcal{U}$ is \textit{point countable on $D\subseteq T$} if $|\U(s)|\leq \omega$ for every $s\in D$. Then a compact $T$ is Valdivia if and only if there are a dense subset $D$ of $T$ and a family $\U$ of open $F_\sigma$ subsets of $T$ which is $T_0$-separating in $T$ and point countable on $D$. Moreover, in this case, $D$ is a $\Sigma$-subset of $T$ (see, for instance, \cite[Proposition 1.9]{K2}).\smallskip

In the proof of Proposition \ref{p: omega1Valdivia} we also need the following lemma.
\begin{lemma}\label{l: closedunbdd}
Let $T$ be a tree endowed with the coarse wedge topology and $t\in T$ be such that $\cf(t)\geq \omega_1$. Assume that there is $S\subseteq T$ such that $t$ is an accumulation point of $S\setminus V_t$. Then $\{s\in T\colon s<t\}$ contains a closed and unbounded set of accumulation points of $S$.
\end{lemma}

\begin{proof} 
We shall show that $\{s\in T\colon s<t\}$ contains an unbounded subset that consists of accumulation points of $S$; then the closure of such subset is the desired closed and unbounded set. For $r\in T$, $r< t$ we denote by $r+1$ the unique element such that $\{r+1\}= \hat{t}\cap \ims(r)$. Fix arbitrarily $r_0<t$. By assumption, we can find $s_0\in V_{r_0+1}\cap (S\setminus V_t)$ and we set $r_1 \coloneqq s_0\wedge t <t$ (since $s_0\notin V_t$). Then we pick $s_1\in V_{r_1+1}\cap (S\setminus V_t)$ and we let $r_2\coloneqq s_1\wedge t <t$. Continuing by induction, we obtain two sequences $(r_j)_{j\in\omega}\subseteq \{s\in T\colon s<t\}$ and $(s_j)_{j\in\omega}\subseteq S\setminus V_t$ with $s_j\in V_{r_j+1}\cap (S\setminus V_t)$ and $r_{j+1}= s_j\wedge t$. The sequence $(r_j)_{j\in\omega}$ is strictly increasing since $r_{j+1}\geq r_j+1>r_j$. Hence, from the assumption that $\cf(t)\geq\omega_1$ we obtain $r\coloneqq \sup r_j <t$. 

We show that $r \in \{s\in T\colon r_0<s<t\}$ is an accumulation point for $S$, which concludes the proof. Take any basic neighbourhood $W_u^F\coloneqq V_u\setminus \bigcup \{V_p\colon p\in F\}$ of $r$, where $u<r$ and $F\subseteq \ims(r)$ is a finite set. Then, there exists $j\in\omega$ such that $u<r_j$, so $s_j\in V_u$. Moreover, $s_j \notin V_p$ for every $p\in \ims(r)$: indeed, $s_j\in V_p$ would imply $s_j\geq p \geq r$, so $r_j=s_j\wedge t\geq r$, a contradiction. Therefore, $s_j\in W_u^F$, hence $r$ is an accumulation point for $S$.
\end{proof}

\begin{proposition}\label{p: omega1Valdivia} Let $T$ be a tree of height at most $\omega_1+1$ endowed with the coarse wedge topology. Then every closed subset of $T$ is a Valdivia compact space.
\end{proposition}

\begin{proof} Let $D=\{t\in T\colon \htte(t,T)<\omega_1\}$. According to \cite[Theorem 4.1]{S2}, $(T,\cw)$ is Valdivia; moreover, \cite[Theorem 2.3]{S3} yields that $D$ is a $\Sigma$-subset for $T$. If $D\cap S$ is dense in $S$, then it is a $\Sigma$-subset for $S$ and $S$ is Valdivia. Therefore, we can assume that the set
\[
S_1\coloneqq S\setminus \overline{D\cap S}
\]
is nonempty. 

We first prove that $S_1$ is a discrete set with the subspace topology. Indeed, suppose by contradiction that $s\in S_1$ is an accumulation point for $S_1$; then $s$ is an accumulation point for $S$ as well. Since $\cf(s)=\omega_1$, by Lemma \ref{l: closedunbdd} $\{r\in T\colon r<s\}\subseteq D$ contains a closed and unbounded set of accumulation points for $S$, hence of points of $S$ ($S$ being closed). This yields that $s\in \overline{D\cap S}$, which is a contradiction since $S_1\cap \overline{D\cap S}=\emptyset$.

Now it remains to show that $S$ is Valdivia. Since $S_1$ is discrete, for every $s\in S_1$ there exists $t(s)\in I(T)$, $t(s)\leq s$ such that
\begin{romanenumerate}
    \item $V_{t(s)}\cap \overline{D\cap S}=\emptyset$;
    \item $V_{t(s)}\cap S_1=\{s\}$.
\end{romanenumerate}
Let $\U_1=\{V_r \colon r\in I(T), r\notin \bigcup_{s\in S_1}V_{t(s)}\}$ and $\U=\{V\cap S\colon V\in \U_1\}\cup\{\{s\}\colon s\in S_1\}$. We shall show that $\U$ and the dense subset $S_1 \cup (D\cap S)$ of $S$ witness that $S$ is Valdivia, by using \cite[Proposition 1.9]{K2} quoted above. First of all, every set in $\U$ is clopen.

Next, we prove that $\U$ is $T_0$-separating on $S$. Pick distinct $s,t\in T$. If $s\in S_1$, then $\{s\}\in\U$ witnesses the $T_0$-separation (and the same if $t\in S_1$). Hence, we can assume $s,t\in \overline{D\cap S}$ and we distinguish two cases. If $s<t$ (and similarly if $t<s$), we pick $r\in I(T)$ with $s< r\leq t$; then $V_r\cap\{s,t\}=\{t\}$. If $s,t$ are incomparable, by chain completeness there is $r\in I(T)$ with $r\leq t$ and such that $s$ and $r$ are incomparable. Hence, $V_r\cap\{s,t\}=\{t\}$. In either case, we have $V_r\cap \overline{D\cap S}\neq\emptyset$, whence $r\notin \bigcup_{s\in S_1}V_{t(s)}$ by (i). Consequently $V_r\cap S\in\U$, showing that $\U$ is $T_0$-separating on $S$.

Let us finally show that $\U$ is point countable on $S_1 \cup (D\cap S)$. If $s\in D\cap S$ and $U\in\U(s)$, then there is $t\leq s$ such that $U=V_t\cap S$. Hence, $\htte(t,T)< \omega_1$ gives that $\U(s)$ is countable. Similarly, let $s\in S_1$ and $U\in \U(s)$; then either $U=\{s\}$ or $U=V_r\cap S$ for some $r< t(s)$. Since $\htte(t(s),T)<\omega_1$, $\U(s)$ is a countable family.
\end{proof}

\begin{remark} The reason of the assumption $\htte(T)\leq\omega_1 +1$ in the above result is that our construction of the family $\U$ only works in this case; on the other hand, the first part of the argument, showing that $S_1$ is discrete extends to Valdivia trees of every height. However, a substantially more technical way to build a $T_0$-separating and point countable family $\U$ on some trees of bigger height was found in \cite[Theorem 3.2]{S3} (see also \cite[Proposition 5.31]{K4}). It seems conceivable that trees as in \cite[Theorem 3.2]{S3} actually are hereditarily Valdivia. However, we decided not to enter such direction, as it would be a rather lengthy detour and we only need Proposition \ref{p: omega1Valdivia} for height $\omega_1+1$.
\end{remark}

\section{The countably coarse wedge topology}\label{s: ccw}

In this section we introduce and study some basic properties of the countably coarse wedge topology on trees which stands in between the coarse wedge topology and the fine wedge topology (we refer to \cite{N1} for details on the fine wedge topology).
\begin{definition} The \textit{countably coarse wedge topology} on a tree $T$, denoted by $\ccw$, is the one whose subbase is the set of all $V_t$ and their complements, where $\cf(t)\neq \omega$.
\end{definition}

In particular, a local base at each point $t\in T$ can be described as follows. If $\cf(t)\neq \omega$ a local base at $t$ is formed by all sets of the form

\begin{equation*}
W_{t}^{F}=V_t\setminus \bigcup\{V_s\colon s\in F\},
\end{equation*}
where $F$ is a finite set of immediate successors of $t$. In case $\cf(t)=\omega$, a local base at $t$ is formed by sets of the form

\begin{equation*}
W_{s}^{F}=V_s\setminus \bigcup\{V_r\colon r\in F\},
\end{equation*}
where $s<t$, $s\in I(T)$, and $F$ is a finite set of immediate successors of $t$.

\begin{proposition}\label{p: basicsigmaprop}
Let $T$ be a tree endowed with the $\ccw$-topology. Then  $(T,\ccw)$ is countably compact and Fr\'{e}chet--Urysohn.
\end{proposition}

\begin{proof} Let $(s_n)_{n\in\omega}\subseteq T$ be an injective sequence. In order to prove that $(T,\ccw)$ is countably compact it is enough to show that $(s_n)_{n\in\omega}$ has a cluster point. Since $T$ endowed with the coarse wedge topology is compact, the sequence $(s_n)_{n\in\omega}$ has a cluster point $s$ in $(T,\cw)$. We show that $s$ is also a cluster point for $(s_n)_{n\in\omega}$ in $(T,\ccw)$. If $\cf(s)\leq \omega$, the local bases at $s$ in $\ccw$ and in the coarse wedge topology coincide, whence $s$ is a $\ccw$-cluster point of $(s_n)_{n\in\omega}$. If $\cf(s)\geq \omega_1$, we argue by contradiction and we assume that $s$ is not a $\ccw$-cluster point of $(s_n)_{n\in\omega}$. Then there exists a finite set $F\subseteq \ims(s)$ such that $W_{s}^{F}\cap (s_n)_{n\in\omega}=\emptyset$. By Fact \ref{f: Vt disjoint from closure}, there is $p\in I(T)$ with $p\leq s$ such that $V_p$ is disjoint from the countable set $(T\setminus V_s)\cap (s_n)_{n\in\omega}$. Then $W_p^F$ is an open neighbourhood of $s$ in the coarse wedge topology and it does not intersect $(s_n)_{n\in\omega}$, a contradiction. This shows that $(T,\ccw)$ is countably compact.\smallskip

Next, we prove that $(T,\ccw)$ is Fr\'{e}chet--Urysohn. Let $A\subseteq (T,\ccw)$ and $t\in \overline{A}\setminus A$. In order to build a sequence $(s_n)_{n\in\omega}$ in $A$ which converges to $t$, we distinguish three cases.
\begin{itemize}
    \item If $\cf(t)\neq \omega$, then $W_t^F\cap A\neq\emptyset$ for every finite set $F\subseteq \ims(t)$. Hence, there exists a countable subset $(t_n)_{n\in\omega}\subseteq \ims(t)$ such that $V_{t_n}\cap A\neq\emptyset$. If we pick $s_n\in V_{t_n}\cap A$, then $(s_n)_{n\in\omega}$ converges to $t$.
    \item If $\cf(t)= \omega$ and there exists a countable subset $(t_n)_{n\in\omega}\subseteq \ims(t)$ with $V_{t_n}\cap A\neq\emptyset$, we pick $s_n\in V_{t_n}\cap A$ as in the previous case.
    \item If $\cf(t)=\omega$ and $F\coloneqq \{s\in\ims(t)\colon  V_{s}\cap A\neq\emptyset\}$ is finite, there exists a sequence $(t_n)_{n\in\omega}\subseteq I(T)$ such that $t_n<t$ and $\sup_{n\in\omega}t_n=t$. Therefore, since $t\in \overline{A}$, we can pick $s_n\in W_{t_n}^F\cap A$ and $(s_n)_{n\in\omega}$ is as desired.
\end{itemize}
Hence, in each case there exists a sequence $(s_n)_{n\in\omega}$ in $A$ which converges to $t$, which proves the assertion.
\end{proof}

Our next result asserts that the countably coarse wedge topology is maximal amongst all countably compact topologies which extend the coarse wedge topology on a tree.

\begin{theorem}\label{t: largestcountablycompact}
Let $T$ be a tree endowed with a topology $\tau$ such that $\ccw \subseteq \tau$. If $(T,\tau)$ is countably compact, then $\ccw=\tau$.
\end{theorem}

\begin{proof} Pick $U\in \tau$ and denote by $\min(U)$ the set of minimal elements in $U$. We first show that, for every $t\in \min(U)$, we have $\cf(t)\neq \omega$. Indeed, suppose on the contrary that $\cf(t)= \omega$. Then there exists an increasing sequence $(t_n)_{n\in \omega}$ such that $t_n<t$, $t_n\in I(T)$ and $\sup_{n\in\omega}t_n=t$. Since $t\in \min(U)$ and $t_n<t$, we have that $t_n\notin U$ for every $n\in\omega$. Hence $(t_n)_{n\in \omega}$ converges to $t$ in $\ccw$, but $t$ is not a cluster point for it in $\tau$. Since $\ccw\subseteq \tau$, we get that $(t_n)_{n\in\omega}$ has no cluster point in $\tau$, which contradicts the countable compactness. Therefore, if $t\in\min(U)$, then $\cf(t)\neq \omega$. 

In particular, if $t\in\min(U)$, then $V_t$ and $T\setminus V_t$ belong to $\ccw$. Hence, $V_t$ is closed in $\ccw$, so then $V_t\setminus U$ is closed in $\tau$. Thus, $V_t\setminus U$ is countably compact in $\tau$, whence in $\ccw$ (using again that $\ccw\subseteq\tau$). As a consequence, $V_t\setminus U$ is $\ccw$-countably closed, hence $\ccw$-closed, since $(T,\ccw)$ is Fr\'{e}chet--Urysohn by Proposition \ref{p: basicsigmaprop}. Thus $V_t\cap U$ is $\ccw$-open. Finally, observing that
\[
U=\bigcup_{t\in\min(U)}V_t\cap U,
\]
we conclude that $U\in\ccw$.\end{proof}

In the second part of this section we show that, when restricting our attention to $r_1$-trees, we can give an explicit connection between $\ccw$ and $\cw$. More precisely, the restriction of the coarse wedge topology to a certain subset $\tilde{T}$ of $T$ is the countably coarse wedge topology on $\tilde{T}$ and vice versa the \v{C}ech--Stone compactification of a tree $T$ with $\ccw$ is homeomorphic to a tree $\hat{T}$ endowed with $\cw$. Since we need to deal with two trees simultaneously, in the remainder of the section we denote by $V_t(T)$ the upwards cone of $t$ in the tree $T$.

\begin{proposition}\label{p: restrictionsigmatopology}
Let $T$ be an $r_1$-tree endowed with the coarse wedge topology. Then the subspace topology on $\tilde{T}=T\setminus \bigcup_{\cf(\alpha) \geq\omega_1}\Lev_{\alpha}(T)$ coincides with the $\ccw$-topology generated by the order $\leq_T$ restricted to $\tilde{T}$. 
\end{proposition}
\begin{proof} We first observe that, since $T$ is an $r_1$-tree, $\tilde{T}$ is chain complete (and rooted) when it is endowed with $\leq_T$. It is easy to see that $I(T)=\{t\in\tilde{T}\colon \cf(t,\tilde{T})\neq\omega\}$. Moreover, for every $t$ in such set we have $V_t(T)\cap \tilde{T}= V_t(\tilde{T})$, hence
\begin{equation*}
    \{\tilde{T}\cap V_t(T)\colon t\in I(T)\}= \{V_t(\tilde{T}) \colon t\in \tilde{T}, \cf(t,\tilde{T}) \neq\omega\}.
\end{equation*}
Therefore, the restriction of $\cw$ to $\tilde{T}$ coincides with $\ccw$ and we are done.
\end{proof}

Let $T$ be a tree such that $\htte(T)\geq\omega_1 +1$. We define a new tree $\hat{T}$ in which, for each $t\in \Lev_{\alpha}(T)$ we add a unique point $s(t)$ between $t$ and its predecessors $\{r\in T\colon r<t\}$, whenever $\cf(\alpha)\geq\omega_1$. We identify $T$ with a subset of $\hat{T}$ in the natural way. Notice that we have $T=\tilde{\hat{T}}$, while $T=\hat{\tilde{T}}$ holds true if and only if $T$ is an $r_1$-tree.
The proof of the following result was suggested to us by the anonymous referee, our previous argument was more complicated.

\begin{theorem}\label{t: compactification}
Let $T$ be a tree. Then $\beta(T,\ccw)$ is homeomorphic to $(\hat{T},\cw)$. 
\end{theorem}

\begin{proof} 
At first we observe that $(T,\ccw)$ is a dense subset of the compact space $(\hat{T},\cw)$. Hence, in order to show that $\beta(T,\ccw)$ is homeomorphic to $(\hat{T},\cw)$ it is enough to show that every pair of disjoint closed subsets of $(T,\ccw)$ has disjoint closures in $(\hat{T},\cw)$. Towards a contradiction, assume that there are disjoint closed sets $A,B$ in $(T,\ccw)$, such that $\overline{A}\cap \overline{B}\neq\emptyset$, where the closure is taken in $(\hat{T},\cw)$. Then there is $t\in T$ such that $\cf(t,T)\geq\omega_1$ and $s(t)\in \overline{A}\cap \overline{B}$, where $s(t)\in \hat{T}$ is the unique element which separates $t$ from $\{r\in T\colon r<t\}$. Using the fact that $\{V_p\setminus V_t\colon p<t\}$ is a local base at $s(t)$, we conclude that $s(t)$ is an accumulation point for $A\setminus V_t$, hence also for $A\setminus V_{s(t)}$ (note that $V_{s(t)}=V_t\cup \{s(t)\}$). Similarly, $s(t)$ is an accumulation point for $B\setminus V_{s(t)}$. According to Lemma \ref{l: closedunbdd}, there exist two closed and unbounded subsets of $\{r\in \hat{T}\colon r<s(t)\}$ that consist of accumulation points of $A$ and $B$ respectively. Hence there also is $s\in\hat{T}$, $s<s(t)$ with $\cf(s,\hat{T})=\omega$ and such that $s$ is an accumulation point both of $A$ and of $B$. (This is a standard argument: pick $s_0<s(t)$ that is an accumulation point for $A$; next pick $s_1$ with $s_0<s_1<s(t)$ and that is an accumulation point for $B$; then take $s_1<s_2<s(t)$ and such that $s_2$ is an accumulation point for $A$. Then $s=\sup s_j$ is as desired.) The condition $\cf(s,\hat{T})=\omega$ implies that actually $s\in T$. Finally, the assumption that $A$ and $B$ are closed in $(T,\ccw)$ gives the contradiction that $s\in A\cap B$.
\end{proof}

Let us conclude the section with one comment on the assumption that $T$ is an $r_1$-tree in Proposition \ref{p: restrictionsigmatopology}. If $T$ is not an $r_1$-tree, then $\tilde{T}$ is not chain complete. On the other hand, according to \cite[Lemma 5.20]{K4}, if $T$ is an $r$-tree, then there is a finer partial order $\preceq$ on $T$ such that the (countably) coarse wedge topologies of $(T,\leq)$ and $(T,\preceq)$ are the same and $(T,\preceq)$ is an $r_1$-tree. Therefore, as it concerns topological properties, considering $r_1$-trees or $r$-trees is essentially equivalent. The situation is different if $T$ is not an $r$-tree, since then $(\tilde{T},\ccw)$ fails to be countably compact. Indeed, if $t\in T$ is such that $\cf(t)\geq\omega_1$ and $\ims(t)$ is infinite, then $\ims(t)$ is a closed, discrete, and infinite subset of $(\tilde{T},\ccw)$. Therefore, $(\tilde{T},\ccw)$ is not countably compact.

\section{Weakly Corson countably compact trees}\label{s: wct}

In the present section we study the relations between the classes of Valdivia compact trees and weakly Corson compact trees. It turns out that a tree $(T,\cw)$ is weakly Corson if and only if $(T,\ccw)$ is Corson countably compact, if and only if $(\hat{T},\ccw)$ is Valdivia. Moreover, if a tree $(T,\cw)$ is weakly Corson, then this is witnessed by the Corson countably compact space $(T,\ccw)$ (see Remark \ref{r: witness weak Corson}). Let us recall that a \textit{quotient map} is an onto map $f\colon X\to Y$ such that $A$ is open in $Y$ if and only if $f^{-1}(A)$ is open in $X$.

\begin{lemma}\label{l: existencequotient}
Let $X$ be a Fr\'{e}chet--Urysohn countably compact space and $T$ be a tree. If $\p\colon X\to (T,\cw)$ is a continuous onto map, then $\p\colon X\to (T,\ccw)$ is a quotient map.
\end{lemma}

\begin{proof} We first prove that $\p^{-1}(V_t)$ is a clopen subset of $X$ for every $t\in T$ with $\cf(t)\neq \omega$. Indeed, if $\cf(t)<\omega$, this is clear since $\p$ is continuous and $V_t$ is $\cw$-clopen. Fix $t\in T$ with $\cf(t)\geq\omega_1$. Then $V_t$ is $\cw$-closed and by continuity of $\p$ it is enough to show that $\p^{-1}(V_t)$ is open in $X$. Towards a contradiction, assume that this is not the case; hence there exists $x\in \p^{-1}(V_t)\cap \overline{X\setminus \p^{-1}(V_t)}$. The assumption that $X$ is Fr\'{e}chet--Urysohn yields the existence of a sequence $(x_n)_{n\in\omega}\subseteq X\setminus \p^{-1}(V_t) $ which converges to $x$. So $(\p(x_n))_{n\in\omega}\subseteq T\setminus V_t$ converges to $\p(x)\in V_t$ in the coarse wedge topology, contrary to Fact \ref{f: Vt disjoint from closure}. This proves that $\p^{-1}(V_t)$ is clopen in $X$ for every $t\in T$ with $\cf(t)\neq \omega$.

Define a topology $\tau$ on $T$ by declaring that $A\subseteq T$ is $\tau$-open if and only if $\p^{-1}(A)$ is open in $X$. Hence, $\p\colon X\to (T,\tau)$ is a quotient map and in particular $(T,\tau)$ is countably compact. Moreover, the first part of the proof gives that $V_t$ and $T\setminus V_t$ belong to $\tau$ for every $t\in T$ with $\cf(t)\neq\omega$; thus $\ccw\subseteq \tau$. By Theorem \ref{t: largestcountablycompact}, $\ccw=\tau$ and we are done.
\end{proof}

As a consequence, we get the following characterisation of weakly Corson compact trees.

\begin{theorem}\label{t: tfaeValdCors}
Let $T$ be a tree. The following are equivalent.
\begin{romanenumerate}
    \item $(T,\cw)$ is a weakly Corson compact space.
    \item $(T,\ccw)$ is a Corson countably compact space.
    \item $(\hat{T},\cw)$ is a Valdivia compact space.
\end{romanenumerate}
\end{theorem}

\begin{proof}
(i)$\Rightarrow$(ii) Since $T$ is weakly Corson,  there exist a Corson countably compact space $X$ and a continuous onto map $\varphi \colon X\to T$. Noting that Corson countably compact spaces are Fr\'{e}chet--Urysohn, Lemma \ref{l: existencequotient} implies that $(T,\ccw)$ is a quotient of $X$. According to \cite[Lemma 2.1]{K1}, the class of Corson countably compact spaces is closed under quotients, so we conclude that $(T,\ccw)$ is a Corson countably compact.\\
(ii)$\Rightarrow$(i) The $\ccw$-topology is finer than the coarse wedge topology on $T$, therefore the identity map $\mathrm{id} \colon (T,\ccw)\to (T,\cw)$ is continuous. Hence, if $(T,\ccw)$ is a Corson countably compact, then $(T,\cw)$ is weakly Corson.\\
(ii)$\Rightarrow$(iii) We mentioned already in the Introduction that the \v{C}ech--Stone compactification of a Corson countably compact is Valdivia (combining \cite[Lemma 1.8]{K2} with \cite[Proposition 1.9]{K2}). So, we obtain that $\beta(T,\ccw)$ is a Valdivia compact space and the assertion then follows from Theorem \ref{t: compactification}.\\
(iii)$\Rightarrow$(ii) By \cite[Theorem 2.3]{S3}, the set $D=\{t\in \hat{T}\colon\cf(t)< \omega_1\}$ is a $\Sigma$-subset of the $r_1$-tree $\hat{T}$. Moreover, Proposition \ref{p: restrictionsigmatopology} yields that $D$ is homeomorphic to $(T,\ccw)$. Therefore, $(T,\ccw)$ is a Corson countably compact.
\end{proof}

\begin{remark}\label{r: witness weak Corson} By Theorem \ref{t: tfaeValdCors}, if $(T,\cw)$ is weakly Corson, this property is witnessed by the Corson countably compact $(T,\ccw)$ and the identity map $\mathrm{id} \colon (T,\ccw)\to (T,\cw)$.
\end{remark}

\begin{remark} The above theorem can be considered the counterpart for trees of \cite[Theorem 3.5]{K1}, asserting that, for an ordinal $\eta$, $[0,\eta]$ is weakly Corson if and only if it is Valdivia (if and only if $\eta<\omega_2$). In light of this result, it is natural to ask whether we can replace $\hat{T}$ with $T$ in (iii) above. However, it turns out that this is not the case. For example, the tree in Example \ref{e: binaryTree} below is Valdivia but not weakly Corson. On the other hand, consider a tree $T$ with $\htte(T)=\omega_1+2$ and such that $|\Lev_\alpha(T)|=1$ if $\alpha\leq\omega_1$ and $|\Lev_{\omega_1+1}(T)| =\omega$. Then the tree $T$ is a weakly Corson compact by Theorem \ref{t: tfaeValdCors}, but it is not Valdivia, since it is not an $r$-tree (see also \cite[Example 1.10]{K2}).
\end{remark}

As an application of the previous results we obtain the following example of a hereditarily (weakly) Valdivia compact space that is not weakly Corson. This answers in the negative \cite[Question 2]{K1}.

\begin{example}\label{e: binaryTree}
Let $T$ be the full binary tree of height $\omega_1 +1$. By Proposition \ref{p: omega1Valdivia}, $T$ is hereditarily Valdivia; in particular, it is hereditarily weakly Valdivia. Then $\hat{T}$ is the subtree of the full binary tree of height $\omega_1+2$ such that $|\ims(t)|=2$ for every $t\in \hat{T}$ with $\htte(t,\hat{T})<\omega_1$, while every $t\in\Lev_{\omega_1}(\hat{T})$ admits a unique immediate successor. By \cite[Example 4.3]{S2}, $(\hat{T},\cw)$ is not Valdivia, therefore by Theorem \ref{t: tfaeValdCors}, $(T,\cw)$ is not a weakly Corson compact space.
\end{example}

\section{Some remarks on \texorpdfstring{$G_\delta$}{G delta} points}\label{s: gdelta}

In this section we show that each weakly Corson compact tree, endowed with the coarse wedge topology, has a dense set of $G_{\delta}$ points. This result is motivated by \cite[Question 1]{K1}, where the following question was posed:
\begin{question}\label{q: Gdelta}
Let $K$ be a weakly Corson compact space. Does $K$ have a dense set of $G_{\delta}$ points?
\end{question}
In \cite[Corollary 3.14]{TT}, Question \ref{q: Gdelta} is given an affirmative answer, under the assumption of the continuum hypothesis. Our result therefore shows that in the subclass of weakly Corson compact trees the result in \cite{TT} is valid inside ZFC; in particular, a possible counterexample must be sought outside the class of weakly Corson compact trees. In this section we only consider the topology $\cw$, also when not explicitly specified.

\begin{proposition}\label{p: Gdelta}
Let $T$ be a tree.  If $(T,\cw)$ is weakly Corson, then $(T,\cw)$ has a dense subset of $G_{\delta}$ points.
\end{proposition}

\begin{proof}
Suppose that $(T,\cw)$ does not have a dense subset of $G_{\delta}$ points. Then there exist an element $t\in I(T)$ and a finite set $F\subseteq \ims(t)$ such that $W_{t}^{F}$ does not contain $G_{\delta}$ points of $(T,\cw)$. Since $t\in T$ is not a $G_{\delta}$ point, then it has uncountably many immediate successors, therefore there exists $s \in \ims(t)\setminus F$. Hence, $V_s$ does not contain $G_{\delta}$ points of $T$. It also follows that $V_s$ contains no $G_\delta$ points for $V_s$ (because $V_s$ is clopen in $T$). Since the class of weakly Corson compact spaces is closed under closed subspaces \cite[Lemma 2.2]{K1}, $(V_s,\cw)$ is a weakly Corson tree as well. (Note that the restriction of the coarse wedge topology of $T$ to $V_s$ is the coarse wedge topology of $V_s$, by \cite[Lemma 2.1]{S3}.) Therefore, up to replacing $T$ with $V_s$ we may suppose without loss of generality that $(T,\cw)$ does not contain $G_{\delta}$ points. Next, we observe that $\htte(T)\geq\omega_1 +1$, or equivalently $T$ contains an $\omega_1$-chain. Indeed, if $T$ does not have $\omega_1$-chains, then by \cite[Theorem 2.8]{N2}, $T$ would be Corson and so it would have a dense subset of $G_{\delta}$ points, \cite[Theorem 3.3]{K2}.

If $\htte(t,T)<\omega_1$, then $|\ims(t)|>\omega$, because otherwise $t$ would be a $G_{\delta}$ point. We now build by transfinite induction a closed subset $S$ of $T$ as follows. We start with $S_0=\{0_T\}$. Fix $\alpha<\omega_1$, and suppose that for each $\beta<\alpha$, $S_{\beta}$ has been defined in such a way that $S_{\beta}$ is a full binary tree of height $\beta+1$, $S_\beta$ is an initial part of $T$, and $S_\gamma$ is an initial part of $S_\beta$ if $\gamma<\beta<\alpha$. If $\alpha=\beta +1$, then $S_{\alpha}\coloneqq S_{\beta}\cup\bigcup_{t\in\Lev_{\beta}(T)}\{s_1^t,s_2^t\}$, where $\{s_1^t,s_2^t\}\subseteq \ims(t)$ and $s_1^t\neq s_2^t$. If $\alpha$ is limit, then $S_{\alpha}=\overline{\bigcup_{\beta<\alpha}S_{\beta}}$. Finally, we set $S=\overline{\bigcup_{\alpha< \omega_1} S_{\alpha}}$. Since $S$ is homeomorphic to the full binary tree of height $\omega_1 +1$ endowed with the coarse wedge topology, that is not weakly Corson by Example \ref{e: binaryTree}, we conclude that $T$ is not weakly Corson.
\end{proof}

Finally we observe that Example \ref{e: binaryTree} also shows that the previous result can not be reversed. Indeed, the full binary tree $T$ of height $\omega_1 +1$ endowed with the coarse wedge topology is not weakly Corson, while on the other hand, $(T,\cw)$ has a dense subset of $G_{\delta}$ points. Indeed, if $t\in T$ and $\htte(t,T)$ is a successor ordinal, then $t$ is an isolated point, therefore it is a $G_{\delta}$ point.

\medskip
{\bf Acknowledgements.} The authors wish to express their gratitude to the anonymous referee for their insightful comments which substantially improved the paper.


\end{document}